\documentclass[11pt]{amsart}

\usepackage{hyperref}
\usepackage{cleveref}
\usepackage[colorinlistoftodos,bordercolor=orange,backgroundcolor=orange!20,linecolor=orange,textsize=scriptsize]{todonotes}

\usepackage{graphicx, enumerate, url}
\usepackage{amssymb,mathtools}

\usepackage{algorithm}
\usepackage{algpseudocode}
\usepackage{color}
\usepackage{tikz}
\usetikzlibrary{3d,calc}

\numberwithin{equation}{section}
\numberwithin{algorithm}{section}

\theoremstyle{plain}
\newtheorem{theorem}{Theorem}[section]

\newtheorem{lemma}[theorem]{Lemma}
\newtheorem{corollary}[theorem]{Corollary}

\theoremstyle{definition}

\theoremstyle{remark}

\DeclareMathOperator{\E}{E}
\DeclareMathOperator{\tr}{Tr}
\DeclareMathOperator{\diag}{diag}
\DeclareMathOperator{\rank}{rank}

\newcommand{\N}{\mathbb{N}}
\newcommand{\C}{\mathbb{C}}
\newcommand{\F}{\mathbb{F}}
\newcommand{\Q}{\mathbb{Q}}
\newcommand{\R}{\mathbb{R}}
\newcommand{\Z}{\mathbb{Z}}

\newcommand{\x}{\mathbf{x}}

\newcommand{\y}{\mathbf{y}}
\newcommand{\z}{\mathbf{z}}
\newcommand{\U}{\mathbf{U}}

\newcommand{\W}{\mathbf{W}}
\newcommand{\cA}{\mathcal{A}}

\newcommand{\cC}{\mathcal{C}}

\newcommand{\cE}{\mathcal{E}}
\newcommand{\cF}{\mathcal{F}}
\newcommand{\cG}{\mathcal{G}}

\newcommand{\cT}{\mathcal{T}}

\newcommand{\bb}{\mathbf{b}}
\newcommand{\bc}{\mathbf{c}}
\newcommand{\be}{\mathbf{e}}

\newcommand{\bi}{\mathbf{i}}

\newcommand{\bq}{\mathbf{q}}

\newcommand{\bs}{\mathbf{s}}

\newcommand{\bu}{\mathbf{u}}
\newcommand{\bv}{\mathbf{v}}

\newcommand{\bx}{\mathbf{x}}

\newcommand{\rA}{\mathrm{A}}
\newcommand{\rB}{\mathrm{B}}

\newcommand{\rH}{\mathrm{H}}

\newcommand{\rK}{\mathrm{K}}
\newcommand{\rL}{\mathrm{L}}

\newcommand{\rS}{\mathrm{S}}

\newcommand{\0}{\mathbf{0}}

\newcommand{\blue}[1]{\textcolor{blue}{#1}}
\makeatletter
\@namedef{subjclassname@2020}{\textup{2020} Mathematics Subject Classification}
\makeatother

\begin{document}
\title[Characterizations of the numerical radius and its dual norm]
{On semidefinite programming characterizations of the numerical radius and its dual norm}
\author{Shmuel Friedland}
\address{
 Department of Mathematics, Statistics and Computer Science,
 University of Illinois at Chicago, Chicago, Illinois 60607-7045,
 USA, \texttt{friedlan@uic.edu}
 }
 \author{Chi-Kwong Li}
\address{Department of Mathematics , College of William and Mary,
Williamsburg, Virginia 23187-8795, USA 
  \texttt{ckli@math.wm.edu}
 }
 \date{January 21, 2024}
\subjclass[2010]{15A60, 15A69, 68Q25, 68W25, 81P40, 90C22, 90C51}
  	
\keywords{
Numerical radius, dual of numerical radius, semidefinite programming, interior point method, polynomial time approximation, spectral norm of 3-tensors.}
\begin{abstract}  
We state and give self contained proofs of semidefinite programming characterizations of the numerical radius and its dual norm  for matrices.  We show that the computation of the numerical radius and its dual norm within $\varepsilon$ precision are polynomially time computable in the data and $|\log \varepsilon |$ using either the ellipsoid method or the short step, primal interior point method.  We apply our results to give a simple formula for the spectral and nuclear norm of $2\times n\times m$ real tensor in terms of the numerical radius and its dual norm.
\end{abstract}
\maketitle

\section{Introduction}\label{sec:intro}
Let $\C^n$ be the complex space of column vectors $\x=(x_1,\ldots,\x_n)^\top$,  and $\|\x\|=\sqrt{\x^*\x}$.  Denote by $\C^{n\times n}\supset\rH_{n}\supset \rH_{n,+}\supset \rH_{n,++}$ the complex space of $n\times n$ complex valued matrices, the real space of $n\times n$ hermitian matrices, the closed cone of positive semi-definite hermitian matrices and the open cone positive definite hermitian matrices.   For $A,B\in\rH_n$ we denote 
$A\succeq B\,(A\succ B)$ if $A-B\in\rH_{n,+}\,(A-B\in\rH_{n,++})$.
For $C\in\C^{n\times n}$ denote by $\sigma_1(C)\ge \cdots \ge \sigma_n(C)\ge 0$ the singular values of $C$.  Then $\|C\|=\sigma_1(C)$ is the operator norm of $C$.  Recall that the dual norm of the operator norm (the nuclear norm) of $C$ is $\|C\|^\vee=\sum_{i=1}^n \sigma_i(C)$ \cite[Proposition 2.1]{RFP10} or \cite[Corollary 7.5.13]{Fri15}.  Clearly,  $\|C\|$ and $\|C\|^\vee$ are convex functions on $\C^{n\times n}$.

The operator norm and its dual are ubiquitous in theoretical and applied mathematics.
A recent paper  of Recht-Fazel-Parrilo \cite{RFP10} discusses the importance of minimization of $\|C\|^\vee$ for a guaranteed minimum-rank solutions of linear matrix equations via nuclear norm.   
The numerical methods discussed in this paper utilizes
the semidefinite programming (SDP) characterizations of $\|C\|$ and $\|C\|^\vee$.

The aim of this paper is to discuss SDP characterizations of the numerical radius of $r(C)$ and its dual norm $r^\vee(C)$ for $C\in\C^{n\times n}$.  Recall that 
\begin{equation}\label{charrvee}
\begin{aligned}
r(C) &=\max\{|\x^*C\x|:\,\x\in\C^n, \|\x\|=1\},\\
r^\vee(C) &=\max_{r(F)\le 1} \Re\tr F^*C,\\
r(C)&=\max_{r^\vee(F)\le 1} \Re\tr F^*C.
\end{aligned}
\end{equation}
See for example \cite[page 3]{FL16} for the definition and a discussion
of the dual norm in general.  Recall that the dual of the dual norm is the original norm.
Note that $\|C\|,\|C\|^\vee$ and $r(C),r^\vee(C)$ are closely related:
\begin{equation*}
\|C\|/2\le r(C)\le \|C\|, \quad
\|C\|^\vee\le r^\vee(C)\le 2\|C\|^\vee.
\end{equation*}
(See \eqref{rCnrmin}.) Perhaps,  a minimization of $\|\cdot\|^\vee$  in \cite{RFP10} can be changed to a minimization of $r^\vee(\cdot)$?)

We now briefly describe the results of our paper.  First we give the SDP characterizations of $r(C)$ and $r^\vee(C)$:
\begin{equation}\label{SDPcharC}
r(C)=\min\left\{c:\,\begin{bmatrix} cI_n +Z&C\\C^*& cI_n -Z\end{bmatrix}\in \rH_{n,+}\right\}.
\end{equation}
\begin{equation}\label{rveeYchar}
r^\vee(C)=\min\left\{\tr X:\, \begin{bmatrix} X& C\\C^*& X\end{bmatrix}\in \rH_{n,+}\right\}.
\end{equation}
The  characterization \eqref{SDPcharC} can be deduced from the result of Ando [Lemma 1]\cite{And73}.    (See also  \cite[Theorem $2. 1$]{Mat93}.)  The characterization \eqref{rveeYchar} can be deduced from \cite[Part (b) of Corollary  4.3]{Mat93}.   We give an independent proof of these two characterizations.

The classical result of  Gr{\"o}tschel-Lov\'asz-Schrijver on the complexity of the SDP problems using the ellipsoid method, or the 
recent results of de~Klerk-Vallentin on the complexity of interior point method, yield that  the computation complexity 
of $r(C)$ and $r^\vee(C)$ within $\varepsilon$ precision are polynomial in entries of $C$ and $|\log \varepsilon|$.  
It is shown in  \cite{FL16} that the computation of a norm in $\R^n$ is NP-hard if and only if the computation of 
its dual norm is NP-hard.)
Theorem \ref{specnrm2mnchar} shows that the spectral and the nuclear norm of tensors in $\R^2\otimes\R^m\otimes\R^n$ is a 
numerical radius and a half of the dual numerical radius respectively of certain $C\in\C^{(m+n)\times (m+n)}$.  Hence, they are polynomially computable.   The $2\times m\times n$ tensor is the simplest example of $3$-tensor that goes beyond the $m\times n$ matrix.
It corresponds to two bilinear forms \cite{Ja79}.  Its rank is related to  multiplication problems in arithmetic complexity.  Its spectral norm is related to the geometric measure of entanglement of qubit--m-dit--n-dit particles \cite[equation (1)]{FW20}.
The nuclear norm of a tensor corresponds to the energy in quantum mechanics \cite{BFZ}.
A recent result  \cite{HJL23} states that for a fixed $l$ the computation complexity 
of the spectral norm in $\R^l\otimes \R^m\otimes \R^n$ is polynomial in $m,n$.
\section{Conic and semidefinite programming}\label{sec:sdpherm}
\subsection{Slater condition for a conic programming problem}\label{subsec:conic}
In this subsection we recall some known results on conic programming we use in this paper, that can be found in \cite[Chapter 4]{GM12}.  
We call $\rK\subset \R^n$  a {\it proper cone} if the following conditions hold:
\begin{equation*}
\begin{aligned}
\textrm{Closure}(\rK)=\rK,\\
a\rK\subseteq \rK \textrm{ for } a\ge 0,\\
\rK+\rK=\rK,\\
\rK\cap -\rK=\{\0\},\\
\rK-\rK=\R^n.
\end{aligned}
\end{equation*}
In the literature, such a closed convex set is also called a pointed generating cone.
Note that $K$ has interior.
Assume $\langle \cdot,\cdot \rangle$ is an inner product on $\R^n$.
Then
$$\rK^\vee=\{\x\in\R^n:\, \langle \x,\y\rangle \ge 0,  \textrm{ for }\y\in\rK\}$$ 
is the dual cone of $\rK$. 

Let $L\in\R^{m\times n}$, and view $L$ as a linear transformation $L:\R^n\to \R^m$.
In $\R^m$ we assume that we have the standard inner product $\langle \bu,\bv\rangle=\bu^\top \bv$.  Denote by $L^\vee:\R^m\to\R^n$ the linear transformation
satisfying:
\begin{equation}\label{defLvee}
\langle L^\vee \y,\x\rangle=\y^\top L\x \textrm{ for all } \x\in\R^n, \y\in\R^m.
\end{equation}

A conic programming problem in $\R^n$ is:
\begin{equation}\label{conprog}
val=\inf\{\langle \bc,\x\rangle :\, L\x=\bb,  \x\in\rK\}, \quad \bc\in\R^n,\bb\in \R^m
\end{equation}
We say that \eqref{conprog} is feasible (infeasible). 
if the set $\{\x\in \rK, L\x=\bb\}$ is nonempty (empty).  If \eqref{conprog} is  infeasible
then $val=\infty$.  

The dual conic problem to \eqref{conprog} is:
\begin{equation}\label{dconprog}
val^\vee=\sup \{\bb^\top\y:\,\bc-L^\vee \y\in\rK^\vee\}
\end{equation}
If the dual conic problem is infeasible, we let $val^\vee=-\infty$.

It is straightforward to show the weak duality inequality $val\ge val^\vee$.
The Slater constraint condition is \cite[Theorem 4.7.1]{GM12}:
\begin{theorem}\label{Slatcon}
Assume that the feasible set of \eqref{conprog} contains an interior point of the cone $\rK$, and $val$ is finite.  
Then the dual problem \eqref{dconprog} is feasible, and $val=val^\vee$.
\end{theorem}

We now make the following observations.  Let us consider the linear  system 
\begin{equation}\label{Lbsys}
L\x=\bb, \quad \x\in\R^n, \bb\in\R^m.  
\end{equation}
Use Gauss elimination to deduce that either the above system in not solvable, hence 
\eqref{conprog} is unfeasible, or the affine set of solutions to the above system is
\begin{equation*}
\{\x= \x_0+\sum_{i=1}^k s_i\x_i\}, \, \x_1,\ldots,\x_k \textrm{ linearly independent in } \R^n,k\in [n]\cup\{0\}.
\end{equation*}
Here $k$ is the dimension of the affine space given by \eqref{Lbsys}.  
Note we can assume that either $\x_0=\0$, or $\x_0,\ldots,\x_k$ are linearly independent.
In particular, we can assume that 
\begin{equation*}
\langle \x_i,\x_j\rangle=\delta_{ij}, i,j\in[k], \quad \langle \x_0,\x_i\rangle=0, i\in[k].
\end{equation*}

We restrict our discussion to the case that the  system \eqref{Lbsys} is solvable, and corank $L=n-\rank L=k$.  
In that case,  the conic programming problem \eqref{conprog} is
\begin{equation}\label{conprog1}
val=\inf\left\{\langle \bc,\x\rangle:\, \x=\x_0+\sum_{i=1}^k s_i \x_i\in \rK,\bs=(s_1,\ldots,s_k)^\top\in\R^k\right\}.
\end{equation}

To give a corresponding simple representation of the dual problem,  complete the orthonormal set $\x_,\ldots,\x_k$ 
to an orthonormal basis $\x_1,\ldots,\x_n$ in $\R^n$.  
Then it is straightforward to show that the dual problem to \eqref{conprog1} is
\begin{equation}\label{dconprog1}
\begin{aligned}
val^\vee=\sup \left\{\langle\x_0,\bu\rangle:\,\bu=\sum_{j=1}^{n-k} t_j\bx_{j+k}, 
\right.\\ \qquad 
\left.
\bc- \bu\in\rK^\vee,(t_1,\ldots,t_{n-k})^\top\in\R^{n-k}\right\}.
\end{aligned}
\end{equation}
\subsection{Semidefinite programming for real and complex matrices}\label{subsec:SDP}
Let $\F$ be the field of real $\R$ or complex numbers $\C$.  Denote by $\rH_n(\F)$ the real space of selfadjoint matrices $\{A\in\F^{n\times n}, A^*=A\}$.  Thus, $\rH_n(\C)=\rH_n$, and $\rH_n(\R)=\rS_n$-the space of real symmetric matrices of order $n$.
Clearly, $\rH_n(\R)\sim\R^{n(n+1)/2}$ and $\rH_n\sim\R^{n^2}$.
The inner product in $\rH_n(\F)$ is $\langle A,B\rangle=\tr AB$.  Thus $\|A\|_F=\sqrt{\langle A, A\rangle}$ is the Frobenius norm of $A$. For $X_0\in \rH_n(\F)$ and $r\ge 0$ denote by $\rB(X_0,r)=\{X\in\rH_n(\F),  \|X-X_0\|_F\le r\}$ the closed ball in $\rH_n(\F)$ centered at $X_0$ with radius $r$.

Let $\rK=\rH_{n,+}(\F)$-the cone of positive semidefinite matrices in $\rH_n(\F)$.
Recall that $\rK^\vee=\rK$.  A conic programming problem discussed in subsection \ref{subsec:conic} is called a semidefinite programming problem.

A standard semidefinite program is:
\begin{equation}\label{srSDP}
val=\inf\left\{\langle A_0,X\rangle :\,X\in \rH_{n,+}(\F), \langle A_j, X\rangle =b_j, j\in [m]\right\},
\end{equation}
where $A_0,A_j\in \rH_n(\F), b_j\in\R$ for $j\in[m]$.
Denote by $\cF$ the feasible set
\begin{equation}\label{defcF}
\cF=\{X\in \rH_{n,+}(\F):\,\langle A_j, X\rangle =b_j, j\in [m]\}.
\end{equation}
Let
\begin{equation*}
\rL(A_1,\ldots,A_m,b_1,\ldots,b_m)=\{X\in\rH_n(\F):\, \langle A_j, X\rangle =b_j, j\in [m]\}.
\end{equation*}
Then $\rL(A_1,\ldots,A_m,b_1,\ldots,b_m)$ is an affine subspace whose dimension is
$k\in\{-1,0\}\cup[\dim \rH_n(\F)]$.  So $k=-1$ if and only if $\rL(A_1,\ldots,A_m,b_1,\ldots,b_m)=\emptyset$, and $k\ge 0$ if $\rL(A_1,\ldots,A_m,b_1,\ldots,b_m)=X_0+\U$ where $$\U=\rL(A_1,\ldots,A_m,0,\ldots,0)\subset \rH_n(\F)$$ is a subspace of dimension $k$.  

By introducing a standard basis in $\rH_n(\F)$ one can use Gauss elimination to determine
the dimension $d$ of $\rL(A_1,\ldots,A_m,b_1,\ldots,b_m)$.  In particular, if $d=\dim \rH_n(\F)-\delta\ge 0$, then $m\ge \delta$, and there exists a subset $\{b_{i_1},\ldots, b_{i_{\delta}}\}$ for some $\{1\le i_1<\ldots <i_{\delta}\}\subset [m]$ such that
\begin{equation}\label{recondL}
\begin{aligned}
\rL(A_1,\ldots,A_m,b_1,\ldots,b_m)=\rL(A_{i_1},\ldots,A_{i_\delta},b_{i_1},\ldots,b_{i_\delta}),  \\
\dim \rL(A_1,\ldots,A_m,b_1,\ldots,b_m)=d=\dim \rH_n(\F)-\delta\ge 0.
\end{aligned}
\end{equation}

As explained in subsection \ref{subsec:conic} we can assume that a standard SDP problem is of the form \cite[Eq. (1)]{VB96}:
\begin{equation}\label{SDPsf}
\begin{aligned}
val=\inf\left\{\sum_{i=1}^k c_i s_i:\, X_0+\sum_{i=1}^k s_i X_i\succeq 0, \hskip .8in \  \qquad \right. \\
\hskip .8in \  \left. X_0,\ldots,X_k\in\rH_n(\F),(s_1,\ldots,s_k)^\top\in\R^k\right\},
\end{aligned}
\end{equation}
(Here,  $X_0,\ldots,X_k$ are fixed matrices.)
Then the dual problem is of the form \cite[Eq. (27)]{VB96}:
\begin{equation}\label{dSDPsf}
val^\vee=\sup\{-\tr X_0 Z:\, \tr X_i Z=c_i, i\in[k],Z\in \rH_{n,+}(\F)\}.
\end{equation}
Theorem \ref{Slatcon} yields
\begin{corollary}\label{corSlater}  Assume that the feasible set of \eqref{SDPsf} contains a positive definite matrix, and $val$ is finite.  Then the dual problem \eqref{dSDPsf} is feasible, and $val=val^\vee$.
\end{corollary}
\subsection{Complexity results for semidefinite programming}\label{subsec:comSDP}
We first recall the known complexity results for real SDP.  
 Assume that $A_i,b_i$ for $i\in [m]$ have rational entries.  Then the problem of 
 
 \medskip
  \qquad
(1)  finding if  $\rL(A_1,\ldots,A_m,b_1,\ldots,b_m)=\emptyset$, or  otherwise

 \medskip
 \qquad
(2)  finding the set $\{b_{i_1},\ldots, b_{i_{\delta}}\}$ that satisfies \eqref{recondL}

\medskip
\noindent
has a simple polynomial complexity  \cite[Sec. 4]{AF21}: $O( n^2m^4 H)$.   Here $H$ is the maximal number of bits needed for each entry of $A_j$ and $b_j$.  (We ignored here the logarithmic factors.)

Without loss of generality we can assume that $0\le m\le n(n+1)/2 -1$, and $\dim \rL(A_1,\ldots,A_m,b_1,\ldots,b_m)=n(n+1)/2-m$.

A polynomial-time bit complexity (Turing complexity) of the SDP problem \eqref{srSDP} for $\F=\R$
 using the ellipsoid method was shown by  Gr{\"o}tschel-Lov\'asz-Schrijver, stated implicitly in \cite[Section 6]{GLS81}.   It used the ellipsoid method of Yudin-Nemirovski \cite{YN76}, and  was inspired by the earlier proof of Khachiyan \cite{Kha79} of the polynomial-time solvability of linear programming.
 
 Let $\Z$, $\Q$, $\Z[\bi]=\Z+\bi \Z$,  and $\Q[\bi]$ the the sets of integers,  rationals, Gaussian integers and rationals respectively.  
The explicit statement of Gr{\"o}tschel-Lov\'asz-Schrijve result is given by 
de~Klerk-Vallentin \cite[Theorem 1.1]{deklerk_vallentin}:
\begin{theorem}\label{GLSelm} Consider the SDP problem \eqref{srSDP} for $\F=\R$.
Assume that $A_j\in \rH_n(\R)\cap \Q^{n\times n}, b_j\in \Q$ for $j\in [m]$,  and  
$\dim\rL(A_1,\ldots,A_m,b_1,\ldots,b_m)=n(n+1)/2-m\ge 1$.
Suppose that there exists $Y_0\in\rH_{n,++}(\R)\cap \Q^{n\times n}$ in the feasible set $\cF$ given by \eqref{defcF} for $\F=\R$ and $0<r\le R\in\Q$ such that 
\begin{equation}\label{BrRcond}
\begin{aligned}
\rL(A_1,\ldots,A_m,b_1,\ldots,b_m)\cap \rB(Y_0,r)\subseteq \cF\\
\subseteq \rL(A_1,\ldots,A_m,b_1,\ldots,b_m)\cap \rB(Y_0,R).
\end{aligned}
\end{equation}
Then for $A_0\in \rH_n\cap \Q^{n\times n}$ and a rational $\varepsilon>0$ one can find $X^\star\in \cF\cap\Q^{n\times n}$ in poly-time using the ellipsoid method such that $\langle A_0,X^\star\rangle-\varepsilon\le val$, where the polynomial is in $n,\log R/r, |\log \varepsilon|$ and the bit size of the data $Y_0,A_0,A_1,\ldots,A_m,b_1,\ldots,b_m$.
\end{theorem}

The complexity of semidefinite programs using the ellipsoid method are discussed in \cite{PK97}.
It was shown by  Karmakar \cite{Kar84} that an IPM is superior to the ellipsoid method for linear programming problems.   
The IPM are discussed in \cite{NN94} and \cite{Ren01}.  In particular, it is shown in  \cite[Theorem 2.4.1]{Ren01} that the number of steps of  the short step primal interior point method is $O\left(\sqrt{\theta(\cF)}\log (\theta(\cF)(R/r)(1/\varepsilon))\right)$.
Here $\theta(\cF)$ is the complexity of the barrier describing $\cF$.

To obtain the Turing complexity results for the above IPM one needs to estimate $\theta(\cF)$, 
and the bit complexity of each short step primal interior point method.  To achieve that one needs to perform Diophantine 
approximations. de~Klerk-Vallentin \cite[Theorem 7.1]{deklerk_vallentin} showed polynomial-time complexity for an SDP problem 
using  IPM method  by applying the short step primal interior point method combined with  Diophantine approximation under 
the condition \eqref{BrRcond}:
\begin{theorem}\label{rdeKVal}(de~Klerk-Vallentin)
Let the assumptions of Theorem \ref{GLSelm} hold.
Then for $A_0\in\rH_n(\R)\cap \Q^{n\times n}$ and rational $\varepsilon>0$ one can find $X^\star\in \cF\cap \Q^{n\times n}$ in poly-time using the short step primal interior point method combined with  Diophantine approximation such that: $\langle C,X^\star\rangle-\varepsilon\le val$, where the polynomial is in $n,\log R/r, |\log \varepsilon|$ and the bit size of the data $Y_0,A_0,A_1,\ldots,A_m,b_1,\ldots,b_m$.
\end{theorem}

We next observe that Theorem \ref{rdeKVal} extends to Hermitian matrices.   Let $Z\in \C^{n\times n}$ and veiw $Z$ as a linear operator on $\C^n$: $\z\mapsto Z\z$.  Write $Z=X+\bi Y, \z=\x+\bi \y$, where $X,Y\in \R^{n\times n}, \x,\y\in \R^n$.  Then the linear operator represented by $Z$ corresponds to linear operator on $\R^n\oplus\R^n$ represented by
$\hat Z=\begin{bmatrix} X&-Y\\Y&X\end{bmatrix}$ acting on $\hat\z=\begin{bmatrix}\x\\\y\end{bmatrix}$.
Observe next that $Z\in \rH_n$ if and only if $X\in \rH_n(\R)$ and $Y\in \rA_n=\{A\in\R^{n\times n}, A^\top =-A$\}.  That is $\hat Z\in\rH_{2n}(\R)$.  Assume $Z\in\rH_n$. Observe that $\z^* Z\z=(\hat \z)^\top \hat Z \hat \z$.   Thus 
$$Z\in \rH_{n,+}(\rH_{n,++})\iff \hat Z\in \rH_{2n,+}(\R)(\rH_{2n,++}(\R)).$$
Finally, for  $A,Z\in\rH_n, $ we have the equality: $\tr A Z=\frac{1}{2}\tr \hat A \hat Z$.

It is worthwhile to mention a recent preprint \cite{Wan23} that gives a more efficient reformulation of a complex SDP as a real SDP.

\section{The dual norm of $r(\cdot)$}\label{sec:dualr}
The following lemma is well known to the experts, and we bring its proof for completeness.
We will let $\sigma_1(A) \ge \cdots \ge \sigma_n(A)$ be the singular values of $A \in \C^{n\times n}$,
and use the fact 
that the dual norm  $\|\cdot\|^\vee$ of the spectral norm, known as the  nuclear norm, is equal to 
$\|A\|^\vee=\sum_{i=1}^n \sigma_i(A)$ \cite[Corollary 7.5.13]{Fri15}.   
\begin{lemma}\label{lcharvee}
Let $C\in\C^{n\times n}$.  Then
\begin{enumerate} [{\rm (a)}]
\item
\begin{equation}\label{trcharr}
\begin{aligned}
r(C) & =\max\{\Re \,e^{-\bi\theta}\x^*C\x:\,\|\x\|=1,\theta\in[0,2\pi)\}\\
& = \max\{\Re \tr C(e^{-\bi\theta}\x\x^*):\,\|\x\|=1,\theta\in[0,2\pi)\}.
\end{aligned}
\end{equation}
\item
The extreme points of the unit ball of the norm $r^\vee(\cdot)$ is the set
\begin{equation}\label{defcE}
\cE =  \{e^{-\bi\theta} \bv\bv^*:\, \theta \in [0, 2\pi), \bv \in \C^n, \|\bv\| = 1\}.
\end{equation} 
\item
\begin{equation}\label{rCchar}
r^\vee(C)=\min\left\{\sum_{i=1}^N \|\x_i\|^2:\, C=\sum_{i=1}^N e^{-\bi\theta_i}\x_i\x_i^*, \x_i\in\C^n, \theta_i\in[0,2\pi)
\right\},
\end{equation}
where $N\le 2n^2+1$. 
\item
\begin{equation}\label{rCnrmin}
\|C\|/2\le r(C)\le \|C\|, \quad
\|C\|^\vee=\le r^\vee(C)\le 2\|C\|^\vee.
\end{equation}
\end{enumerate}
\end{lemma}

\begin{proof} 
{\rm (a)} Assume that $z\in\W(C)$.  Then $z=\x C\x^*$ for some $\x\in\C^n, \|\x\|=1$.  Hence,  
$$|z|=\max\{ \Re \,e^{-\bi\theta}\x^*C\x:\,\theta\in[0,2\pi)\}.$$
This proves \eqref{trcharr}.

\noindent (b)  
Let $\cE$ be the set of the extreme points of the unit ball of $r^\vee(\cdot)$
Clearly, the first characterization of $r(C)$ in \eqref{charrvee} can be replaced by $r(C)=\sup_{F\in\cE}\Re\, F^*C$.  
The second part of the characterization  \eqref{trcharr} yields that 
$$\cE\subseteq  \{e^{-\bi\theta} \bv\bv^*:\, \theta \in [0, 2\pi), \bv \in \C^n, \|\bv\| = 1\}.$$
Observe next that $\|e^{-\bi\theta} \bv\bv^*\|_F=1$.   Recall that the extreme points of the Euclidean norm $\{G\in\C^{n\times n}:\, \|G\|_F\le 1\}$ is the unit sphere $\{G\in\C^{n\times n}:\, \|G\|_F=1\}$.  Hence, equality \eqref{defcE} holds.

\noindent
(c) Clearly,  \eqref{rCchar} holds for $C=0$, where $N=1$.  Assume that $C\ne 0$.
By considering $C_1=\frac{1}{r^\vee(C)}C$ it is enough to consider the case $r^\vee(C)=1$.    Observe that $r^\vee(e^{-\bi\theta}\x\x^*)=1$ if $\|\x\|=1$.  Use the homogeneity of $\|\x\|^2$ to deduce that $r^\vee(e^{-\bi\theta}\x\x^*)=\|\x\|^2$.  Assume that $C=\sum_{i=1}^N e^{-\bi\theta_i}\x_i\x_i^*$.  As $r^\vee(\cdot)$ is a norm we deduce
\begin{equation*}
r^\vee(C)=r^\vee(\sum_{i=1}^N e^{-\bi\theta_i}\x_i\x_i^*)\le \sum_{i=1}^n r^\vee(e^{-\bi\theta_i}\x_i\x_i^*)=\sum_{i=1}^n \|\x_i\|^2.
\end{equation*}

As the real dimension on $\C^{n\times n}$ is $n^2$, Caratheodory's theorem yields that there exists $N\in[n^2+1]$ such that
\begin{equation*}
\begin{aligned}
C=\sum_{i=1}^N c_ie^{-\bi\theta} \bv_i\bv_i^*, \, \|\bv_i\|=1, c_i>0, i\in[N], \sum_{i=1}^N c_i=1\Rightarrow\\
1=r^\vee(C)\le \sum_{i=1}^N c_ir^\vee(e^{-\bi\theta} \bv_i\bv_i^*)=\sum_{i=1}^N c_i=1\Rightarrow\\
C=\sum_{i=1}^N e^{-\bi\theta} (\sqrt{c_i}\bu_i)(\sqrt{c_i}\bu_i)^*,  \quad r(C)=\sum_{i=1}^n \|\sqrt{c_i}\bu_i\|^2.
\end{aligned}
\end{equation*}
This shows \eqref{rCchar}.

\noindent
(d) See \cite[(g),  page 44]{HJ91} for the first inequality of \eqref{rCnrmin}.   
Use the characterization of the dual norms $\|\cdot\|^\vee$ and $r^\vee(\cdot)$ (as in \eqref{charrvee}) and the the first set of inequalities of \eqref{rCnrmin} to deduce the second set of inequalities of \eqref{rCnrmin}.
\end{proof}

The following result was obtained in \cite{Mat93} (Theorem 4.2 and Corollary 4.3)
 using a general cone theoretic argument. We give a direct proof of it.
\begin{theorem}\label{charrvub}
Let $Y \in \C^{n\times n}$.   The following conditions are equivalent:
\begin{enumerate} [(a)]
\item
There is $X \in \rH_{n,+}$ with $\tr X = 1$ such that
$\begin{bmatrix} X & Y \cr Y^* & X \cr\end{bmatrix}$ is positive semidefinite. 
\item The inequiality $r^\vee(Y) \le 1$ holds.
\end{enumerate}
\end{theorem}
\begin{proof} $(b)\Rightarrow (a)$.
As the extreme points of the unit ball of the norm $r^\vee(\cdot)$ is the set \eqref{defcE}, it follows that $r^\vee(Y)\le 1$ if and only if $Y$ is in the convex hull of $\cE$. 
Assume that $Y = \sum_{j=1}^r p_j e^{\bi\theta_j} \bv_j\bv_j^*$, where $\|\bv_j\|=1, \theta_j\in\R, j\in[r]$ and $(p_1, \dots, p_r)^\top$ is a probability vector. Then for $X=\sum_{j=1}^r p_j\bv_j\bv_j^*$ the matrix
$$\begin{bmatrix} X & Y \cr Y^* & X \cr\end{bmatrix}
= \sum_{j=1}^r p_j\begin{bmatrix} \bv_j\bv^*_j 
& e^{\bi\theta_j} \bv_j\bv_j^* \cr e^{-\bi\theta_j} \bv_j\bv_j^* & \bv_j \bv_j^*\end{bmatrix}$$ is positive semidefinite, and $\tr X=1$.  

$(a)\Rightarrow (b)$.  Suppose $X\in \C^{n\times n}$ has trace 1 and 
$T = \begin{bmatrix} X  & Y \cr Y^* & X \cr \end{bmatrix}$ is positive semidefinite.
It is enough to consider the case where $Y\ne 0$.  Hence $0\ne X\succeq 0$.
Assume that $X = U^*DU$ such that $U$ is unitary  and
$D = \diag(d_1, \dots, d_n)$ with $d_1 \ge \cdots \ge d_n \ge 0$. 
By replacing $(X,Y)$ with $(U^*XU, U^*YU)$ we can assume that $X=D$.
Note that if $d_n=\cdots =d_{k+1}=0<d_k$,  the assumption that $T\succeq 0$ yields that the last $n-k$ rows and columns of $Y$ are zero.  Hence, to prove that $Y\in$conv$\,\cE$ it is enough to assume that $d_n>0$ and $\sum_{i=1}^n d_i=1$.

Let 
$$T_1=\begin{bmatrix}D^{-1/2}&0\\0&D^{-1/2}\end{bmatrix}\begin{bmatrix}D&Y
\\Y^*&D\end{bmatrix}\begin{bmatrix}D^{-1/2}&0\\0&D^{-1/2}\end{bmatrix}=\begin{bmatrix}I&C\\C^*&I\end{bmatrix} \succeq 0.$$
As $T_1$ is positive semidefinite, $C=D^{-1/2}YD^{-1/2}$ is a contraction. 
Suppose $C$ has singular value decomposition 
$\sum_{j=1}^n c_j \bu_j\bv_j^*$ with $1 \ge c_1 \ge \cdots \ge  c_n \ge 0$
and orthonormal sets $\{\bu_1, \dots, \bu_n\}, \{\bv_1, \dots, \bv_n\} \subseteq \C^n$.
Let $s_j = \sqrt{1-c_j^2}$ for $j = 1,\dots, n$, 
$C_1 = \sum_{j=1}^n (c_j + \bi s_j)\bu_j\bv_j^*$ and $C_2 = \sum_{j=1}^n (c_j - \bi s_j) \bu_j\bv_j^*.$ 
Then $C = (C_1 + C_2)/2$.
We will show that $Y_j$ is in conv$\,\cE$, the  convex hull 
of $\cE$ for $j = 1,2$. Then $Y \in$conv$\,\cE$.

Note that all singular values of $C_j$ are equal 1 so that $C_j$ is unitary.   It is enough to show that $C_1$ in conv$\,\cE$. There is a unitary 
$Q_1$ such that $C_1 = Q_1 D_1 Q_1^*$ with $D_1 = \diag(e^{\bi \theta_1}, \dots, e^{\bi\theta_n})$.   
Observe the equality
$$C_1=\sum_{j=1}^n e^{\bi\theta_j}\bq_j\bq_j^*,  \quad Q_1=[\bq_1,\ldots,\bq_n].$$
Then 
$$Y_1 = D^{1/2}Q_1D_1 Q_1^*D^{1/2}= \sum_{j=1}^n e^{\bi\theta_j}(D^{1/2}\bq_j)(D^{1/2}\bq_j)^*.$$
It is left to show that 
$\sum_{j=1}^n \|D^{1/2}\bq_j\|^2=1$, which holds as
$$\sum_{j=1}^n \tr D^{1/2}\bq_j \bq_j^* D^{1/2} =\sum_{j=1}^n \tr \bq_j  \bq_j^*D=\tr(\sum_{j=1}^n\bq_j \bq^*_j)D=\tr D=1.$$ 
\end{proof}
\begin{theorem}\label{unbrvee} Let $C\in\C^{n\times n}$.   Then \eqref{rveeYchar} holds.
\end{theorem}
\begin{proof}
Let $r^\vee(C)=t\ge 0$.  If $C=0$ then  \eqref{rveeYchar}  trivially holds. Assume that $t>0$, and let $C_1=\frac{1}{t}C$.  Note that $r^\vee(C_1)=1$.  Hence, it is enough to show \eqref{rveeYchar} for $C_1$.  Theorem \ref{charrvub} implies that there exists $X \in \rH_{n,+}$ with $\tr X = 1$ such that
$\begin{bmatrix} X & C_1 \cr C_1^* & X \cr\end{bmatrix}\succeq 0$.   Assume to the contrary to  \eqref{rveeYchar} that there exists $X_1\in\rH_{n,+}, \tr X_1<1$ such that  
$\begin{bmatrix} X_1 & C_1 \cr C_1^* & X_1 \cr\end{bmatrix}\succeq 0$.  As $C_1\ne 0$ we deduce that $X_1\ne 0\Rightarrow \tr X_1=t_1\in (0,1)$.   Let $X_2=\frac{1}{t_1} X_1, C_2=\frac{1}{t_1} C_1$.  Then $\begin{bmatrix} X_2 & C_2 \cr C_2^* & X_2 \cr\end{bmatrix}\succeq 0$.  As $\tr X_2=1$, Theorem \eqref{charrvub} yields that $r^\vee(C_2)\le 1\Rightarrow r^\vee(C_1)\le t_1<0$, which contradicts the equality $r^\vee(C_1)=1$.
\end{proof}
\begin{theorem}\label{SDPrC}  Let $C\in\C^{n\times n}$.   Then characterization \eqref{SDPcharC} holds.
\end{theorem}
\begin{proof}  Consider the infimum problem
\begin{equation*}
\mu(C)=\inf\{c:\,\begin{bmatrix} cI_n +Z&C\\C^*& cI_n -Z\end{bmatrix}\succeq 0\}.
\end{equation*}
Let $T(c,Z)=\begin{bmatrix} cI_n +Z&C\\C^*& cI_n -Z\end{bmatrix}$.
Observe that $T(c,Z)\succeq 0$ yields
\begin{equation*}
\begin{aligned}
cI+Z\succeq 0, \quad cI-Z\succeq 0\Rightarrow c\ge 0, Z\in\rH_n,\\
(\x^*, -e^{-\bi\theta}\x^*)T(c,Z)(\x^*, -e^{-\bi\theta}\x^*)^*\ge 0\Rightarrow  2(\x^*\x c-\Re  e^{\bi\theta}\x^* C\x)\ge 0.
\end{aligned}
\end{equation*}
Hence $c\ge r(C)$.   Clearly, for $\varepsilon>0$ the following condition hold:
$$T(\|C\|+\varepsilon,0)=(\|C\|+\varepsilon) I_{2n}+S(C)\succ 0.$$ 
Hence $\mu(C)\le \|C\|$.

Observe that the infimum problem for $\mu(C)$ is a standard SDP problem of the form
\eqref{SDPsf}:
\begin{equation}\label{SDPmin}
\begin{aligned}
\inf\left\{z_0:\, S(C)+ z_0I_{2n}+\sum_{i=1}^n  z_{ii}A_{ii}+ \hskip 1.5in \right. \
\\ \left. \hskip .7in \sum_{1\le i <j\le n} (\Re z_{ij} A_{ij}+\Im z_{ij} B_{ij})\succeq 0, Z=[z_{ij}]\in\rH_n\right\},
\end{aligned}
\end{equation}
where
\begin{equation*}
\begin{aligned}
A_0 &=I_{2n}, \quad A_{ii}=\begin{bmatrix}\be_i\be_i^\top&0\\0&-\be_i\be_i^\top\end{bmatrix}, \quad i\in[n],\\
A_{ij} & =\begin{bmatrix}\be_i\be_j^\top+\be_j\be_i^\top&0\\0&-\be_i\be_j^\top-\be_j\be_i^\top\end{bmatrix}, \quad 1\le i< j\le n,\\
B_{ij} & =\bi\begin{bmatrix}\be_i\be_j^\top-\be_j\be_i^\top&0\\0&-\be_i\be_j^\top+\be_j\be_i^\top\end{bmatrix}, \quad
1\le i< j\le n.
\end{aligned}
\end{equation*}
Recall that the dual problem to \eqref{SDPmin} is \eqref{dSDPsf}:
\begin{equation*}
2\sup\left\{-\Re\tr (CY^*):\,\begin{bmatrix} X& Y\\Y^*& X\end{bmatrix}\succeq 0,   \tr X=1/2\right\},
\end{equation*}
which is equivalent to 
\begin{equation*}\label{SDPmax}
\sup\left\{-\Re\tr (CY^*):\,\begin{bmatrix} X& Y\\Y^*& X\end{bmatrix}\succeq 0,   \tr X=1\right\}.
\end{equation*}
Use Corollary \ref{unbrvee} and the maximum characterization of 
$r^\vee$ in Theorem \ref{charrvub} to deduce that the above supremum is $r(C)$.
As the admissible set for $\mu(C)$ contains a positive definite matrix Corollary \ref{corSlater} yields that $\mu(C)=r(C)$.
\end{proof}

In the finite dimensional case,
Lemma 1 in  \cite{And73} states that $r(C)\le 1$ if and only if there exists $Y\in\rH_n$ that satisfies the condition $\begin{bmatrix}2(I-Y)&C\\C^*&2Y\end{bmatrix}\in\rH_{2n,+}$.  Set $Z=I-2Y$ and deduce this lemma from
Theorem \ref{SDPrC}.

\section{The computation of $r(C)$ and $r^\vee(C)$ is poly-time}\label{sec:comprC}
\subsection{Poly-time complexity of $r(C)$, $r^\vee(C)$, $\|C\|$ and $\|C\|^\vee$ using SDP}\label{subsec:ptrC}
\begin{theorem}\label{ptcrC}  Let $C\in \Q^{n\times n}[\bi]$. and $0<\varepsilon\in\Q$.
Then there exists an $\varepsilon$ approximation of $r(C), r^\vee(C), \|C\|,\|C\|^\vee$ in poly-time in 
$n, |\log \varepsilon|$ and the entries of $C$ using the ellipsoid method or the short step primal interior point method combined with  Diophantine approximation.
\end{theorem}
\begin{proof}  
 We can find  $\omega(C)\in\N$ in polynomial time in $n$ and  the entries of $C$   such that $\|C\|_F+1\le \omega(C)\le \|C\|_F+2$.   We first consider $r(C)$.  Recall that $r(C)\le \|C\|\le \|C\|_F$.   We next consider the following subset of hermitian matrices in $\C^{(2n+1)\times (2n+1)}$:
\begin{equation}\label{defcA}
\begin{aligned}
\cA=\{X=\begin{bmatrix}aI_n +Z&C&0\\C^*&aI_n-Z&0\\0&0&t\end{bmatrix}\in \rH_{2n+1}:\,\\ 2na+t=3(2n+1)\omega(C)\}.
\end{aligned}
\end{equation}
It is straightforward to show that 
\begin{equation*}
\begin{aligned}
\cA=L(A_1,\ldots,A_m,b_1,\ldots,b_m),  m=(2n+1)^2-(n^2+1),  \\
\dim L(A_1,\ldots,A_m,b_1,\ldots,b_m)=n^2+1,
\end{aligned}
\end{equation*}
for some $A_1,b_1,\ldots, A_m,b_m$. (See the proof of Theorem \ref{SDPrC}.)
Let $\cF$ be the admissible set $\cF=\cA\cap \rH_{2n+1,+}$.
Set $F=\frac{1}{2n}\diag(I_{2n},0)$.  Using \eqref{SDPcharC}, one can  show that
\begin{equation*}
r(C)=\min\{\langle F,X\rangle:\,X\in \cF\}.
\end{equation*}
It is left to show that the conditions of Theorems \ref{GLSelm} and  \ref{rdeKVal} are satisfied.
Clearly,  we can assume that $n\ge 2$.

Let 
\begin{equation*}
Y_0=\begin{bmatrix}3\omega(C)I_n&C&0\\C^*&3\omega(C)I_n&0\\0&0&3\omega(C)\end{bmatrix}=\diag(3\omega(C)I_{2n}+S(C),3\omega(C)),
\end{equation*}
where 
\begin{equation}\label{defS(F)} 
S(F)=\begin{bmatrix}0&F\\F^*&0\end{bmatrix}\in \rH_{m+n}\textrm{ for } F\in\C^{m\times n}. 
\end{equation} 
 Let $r=\rank C$.  Then $\rank S(C)=2r$, and the nonzero eigenvalues of $S(C)$ are $\pm $ of the nonzero singular values of $C$ \cite{Fri15}.
Thus 
$$\lambda_{\max}(S(C))=\|C\|\ge \cdots\ge \lambda_{\min}(S(C))=-\|C\|, \quad \|S(C)\|=\|C\|.$$
Hence, $\lambda_{2n+1}(Y_0)> 2\omega(C)$. In particular, $Y_0$ is positive definite, and $Y_0\in\cF$.  We next show that $\cF\supset \cA\cap \rB(Y_0,\omega(C))$.  

Assume that $X\in \cA\cap \rB(Y_0,\omega(C))$.  So $X$ is of the form given by \eqref{defcA}.  Hence 
\begin{equation*}
\begin{aligned}
X-Y_0 & =\diag((a-3\omega(C))I_n+Z,(a-3\omega(C))I_n-Z, t-3\omega(C),  \\
& 2n a+t=3(2n+1)\omega(C), \quad \|X-Y_0\|_F\le \omega(C).
\end{aligned}
\end{equation*}
Recall that for any $G\in\C^{p\times q}$ one has inequality $\|G\|\le \|G\|_F$.
Therefore one has the inequalities  

\medskip\noindent
(1) \ $\omega(C)\ge |t-3\omega(C)|\Rightarrow t\ge 2\omega(C),$

\medskip\noindent
(2) \ $ \omega^2(C)\ge \|(a-3\omega(C))I_n+Z\|_F^2 +\|(a-3\omega(C))I_n-Z\|_F^2 $

\medskip \hskip .6in
$=
2(n(a-3\omega(C))^2+\|Z\|_F^2)$
$$\Rightarrow 
\omega \ge \sqrt{2n}|a-3\omega(C)|\ge 2|a-3\omega(C)|\Rightarrow a\ge \frac{5}{2}\omega(C),$$

\medskip\noindent
(3) \
$
\omega(C)\ge \sqrt{2}\|Z\|_F\ge \sqrt{2}\|Z\|= \sqrt{2}\|\diag(Z,-Z)\|\ge -\sqrt{2}\lambda_{\min}(\diag(Z,-Z)).$

\medskip\noindent
Hence,
\begin{equation*}
\begin{aligned}
\lambda_{\min}(aI_{2n}+\diag(Z,-Z)+S(C))=a+\lambda_{\min}(\diag(Z,-Z)+S(C))\\
\ge a-\frac{1}{\sqrt{2}}\omega(C)-\|C\|> \frac{3}{4}\omega(C)\Rightarrow \lambda_{\min}(X)> \frac{3}{4}\omega(C).
\end{aligned}
\end{equation*}

We claim that $\cF\subset \cA\cap \rB(Y_0, 8n\omega(C))$.  Assume that $X\in\cF$.  So $X\succeq 0$ is of the form given by $\eqref{defcA}$.  Hence $aI_n+Z\succeq 0, aI_n-Z\succeq 0, t\ge 0$.  As $2na +t=3(2n+1)\omega(C)$ we deduce that 
$$0\le t\le 3(2n+1)\omega(C), \quad 0\le a\le \frac{3(2n+1)}{2n}\omega(C)\le \frac{15}{4}\omega(C)<4\omega(C),$$
$$\|Z\|\le a\le 4\omega(C)\Rightarrow \|Z\|_F^2\le 16\omega^2(C)n.$$
Hence,
$$ \|X-Y_0\|_F^2=\|(a-3\omega(C))I_n+Z\|_F^2+\|(a-3\omega(C)I_n-Z\|_F^2 \\+(t-3\omega(C))^2 $$

\medskip\hskip .6in
$=2n(a-3\omega(C))^2+2\|Z\|_F^2+|t-3\omega(C)|^2$

\medskip\hskip .6in
$\le 
(18n+32n+36n^2)\omega^2(C)<64n^2\omega^2(C).$

\medskip\noindent
Observe that $\frac{R}{r}=\frac{8n\omega(C)}{\omega(C)}=8n$.  One can then use Theorems \ref{GLSelm} and  \ref{rdeKVal} to conclude the proof for $r(C)$.

Consider now $r^\vee(C)$.  Let 
\begin{equation*}
\cA=\left\{Z=\begin{bmatrix}X&C&0\\C^*&X&0\\0&0&t\end{bmatrix}\in \rH_{2n+1}:\,
X \in \rH_n, \ \tr Z=(4n+2)\omega(C)\right\}.
\end{equation*}
It is straightforward to show that 
$$\dim \cA = n^2, \quad \cA=\rL(A_1,\ldots,A_{(n+1)(3n+1)},b_1,\ldots,b_{(n+1)(3n+1)}),$$
$$A_j \in \rH_{2n+1}\cap\Q^{(2n+1)\times (2n+1)}[\bi], \quad b_j\in \Q, j\in[(n+1)(3n+1)].$$
Then $\cF=\cA\cap \rH_{2n+1,+}$.  The characterization \eqref{rveeYchar} yields that 
$$r^\vee(C)=\min_{Z\in \cF} \tr \diag(I_{n},0)Z.$$
Let $Z_0=\diag(2\omega(C)I_{2n}+S(C), 2\omega(C))$.  Clearly,  $\lambda_{\min}(Z_0)\ge \omega(C)$.  Hence, $Z_0\in \cF$.  The arguments for the case $r(C)$ yield that 
$\rB(Z_0, \omega(C))\cap \cA\subset \cF$.  Similarly, it follows that $\cF\subset \cA\cap \rB(Z_0,5n\omega(C))$.  Hence $\frac{R}{r}=5n$. 
One can then use Theorems \ref{GLSelm} and  \ref{rdeKVal} to conclude the proof for $r^\vee(C)$.

Consider now $\|C\|$.  Clearly
\begin{equation}\label{SDP||C||}
\|C\|=\min\{a:\,a I_{2n}+S(C)\succeq 0\}.
\end{equation}
Set $$\cC=\{Z=\diag(aI_{2n}+S(C),t):\,2na+t=(4n+2)\omega(C)\}.$$ 
Then $\cF=\cC\cap \rH_{2n+1,+}$.   Let $Z_0=\diag(2\omega(C)I_{2n}+S(C), 2\omega(C))$.  
As in the case of $r^\vee(C)$ it follows that 
$$\rB(Z_0,\omega)\cap \cC\subset \cF\subset \rB(Z_0,5n\omega(C)).$$

Consider now $\|C\|^\vee$.   It was shown in \cite{Mat93} that $\|C\|^\vee\le 1$ if and only if 
there exists a positive semi-definite matrix $Q=\begin{bmatrix}X&C\\C^*&Y\end{bmatrix}$ such that $\tr X=\tr Y=1$. 
We give a short proof of this fact in the following.
Recall that the extreme points of the nuclear norm are $\bu\bv^*, \|\bu\|=\bv\|=1$ \cite[Proof of Theorem 7.5.9]{Fri15}.
If $\|C\|^\vee \le 1$, then $C = \sum_{j=1}^r p_j \bu_j \bv_j^*$ for some positive 
numbers $p_1, \dots, p_r$ summing up to one, and unit vectors  $\bu_1, \dots, \bu_r, \bv_1, \dots, \bv_r$.
Let $X = \sum_{j=1}^r p_j \bu_j \bu_j^*$ and $Y = \sum_{j=1}^r p_j \bv_j \bv_j^*$. Then 
$Q = \begin{bmatrix}X&C\\C^*&Y\end{bmatrix} = \sum_{j=1}^r p_j z_j z_j^*$ is positive semidefinite,
where $z_j = \begin{bmatrix} x_j\cr y_j\cr\end{bmatrix}$ for $j = 1, \dots, r$.
Conversely, suppose that $Q\succeq 0$.  Let $C=\sum_{i=1}^n \sigma_i(C)\bu_i\bv_i^*$ be a singular value decomposition of $C$. Thus $\bu_1,\ldots,\bu_n$ and $\bv_1,\ldots,\bv_n$ are two orthonormal bases in $\C^n$.  Observe that 
$$0\le \sum_{i=1}^n (\bu_i^*, -\bv_i^*) Q(\bu_i^*, -\bv_i^*)^*=2(1-\|C\|_1).$$
Hence, $\|C\|^\vee\le 1$.
Thus, $\|C\|^\vee$ has the following SDP characterization:
\begin{equation}\label{SDPnucnrm}
\|C\|^\vee=\min\{\tr \diag(I_n,0) V:\, Q=\begin{bmatrix}X&C\\C^*&Y\end{bmatrix}\succeq 0,\tr X=\tr Y\}.
\end{equation}
One can use similar arguments  for the case $r^\vee(C)$ to deduce the theorem for $\|C\|^\vee$. 
\end{proof}

A recent paper of T.  Mitchell \cite{Mit23} describes the best known methods for computing $r(C)$.  It also mentions \cite[page A754]{Mit23} that the SDP algorithm suggested in \cite{Mat93} is expensive.   In \cite{MO23} the authors study experimental comparison of methods for computing the numerical radius.  Their results show that the SDPT3 and SeDuMi solvers are much slower than Level Set with Optimization Algorithm 3.1 in \cite{Mit23}.  
\section{Applications}\label{sec:appl}
\subsection{A known characterization of $r(C)$}\label{subsec:knwnr(C)}
The following characterization of $r(C)$ is well known \cite{Kip52}, \cite[page 41]{HJ91}:
\begin{lemma}\label{knwnr(C)}
Let $C=A+\bi B$, where $A,B\in \rH_n$.  Then
\begin{equation*}
r(C)=\max_{\theta\in[0,2\pi)} \lambda_{\max}(\cos\theta A+\sin\theta B).
\end{equation*}
\end{lemma}
Combine the above characterization with convexity of $\lambda_{\max}(\cdot)$ on $\rH_n$ and the characterization \eqref{SDPcharC} to deduce
\begin{corollary}\label{maxcharpenc} Let $A,B\in \rH_n$.  Then
\begin{equation*}
\max\left\{\lambda_{\max}(xA+yB):\,x^2+y^2\le 1\right\}
=\min\left\{c:\,\begin{bmatrix} cI_n +Z&A+\bi B\\A-\bi B& cI_n -Z\end{bmatrix}\succeq 0\right\}.
\end{equation*}
\end{corollary}

\subsection{Application to computing the spectral and nuclear norms of $2\times m\times n$ real tensor}\label{subsec:tspecnrm}
Let $\F$ be either $\R$ or $\C$.
Denote by 
$$\Sigma(\F^l)=\{\x\in \F^l, \|\x\|=1\}$$ the unit sphere in $\F^l.$
Let  $\F^l\otimes\F^m\otimes \F^n$ be the space of $3$-tensors $\cT=[t_{ijk}], i\in[l],j\in[m],k\in[n]$.    Then $\cT(\x,\y,\z)$ is the trilinear form $\sum_{i=j=k=1}^{l,m,n} t_{ijk}x_iy_jz_k$.
The spectral norm of $\cT$ over $\F$ is defined as
\begin{equation*}
\|\cT\|_\sigma=\max\{|\cT(\x,\y,\z)|:\,\x\in\Sigma(\F^l), \y\in\Sigma(\F^m),\z\in\Sigma(\F^n)\}.
\end{equation*}
For $i \in [l]$, let 
\begin{equation}\label{defFi}
F_i=[t_{ijk}],  \quad j\in[m], k\in[n]\in\F^{m\times n}.   
\end{equation}
It is straightforward to see \blue{that for $x = (x_1, \dots, x_l)^t \in \F^l$,}
\begin{equation*}
\|\sum_{i=1}^l x_i F_i\|=\max\{|\cT(\x,\y,\z)|:\,\y\in \Sigma(\F^m),\z\in\Sigma(\F^n)\}.
\end{equation*}
Hence
\begin{equation}\label{maxeq}
\|\cT\|_\sigma=\max\left\{\|\sum_{i=1}^l x_i F_i\|: \x\in\Sigma(\F^l)\right\}.
\end{equation}

The dual of the spectral norm is $\|\cdot\|_\sigma^\vee$ is the nuclear norm denoted as $\|\cdot\|_*$ \cite{FL18}:
\begin{equation}\label{nucchar}
\|\cT\|_*=\min\left\{\sum_{i=1}^N \|\x_i\|\|\y_i\|\|\z_i\|:\, \cT=\sum_{i=1}^N \x_i\otimes \y_i\otimes \z_i,  \x_i\in\F^l,\y_i\in\F^m,\z_i\in\F^n\right\}.
\end{equation}
Let $\dim_\R\F^l=l,  \dim_\R\C^l=2l$.
Caratheodory's theorem yields that one can choose 
$N=(\dim_\R\F^l)(\dim_\R\F^m)(\dim_\R\F^n)+1$.
In quantum mechanics,  $\|\cT\|_*$ measure the energy of $\cT$ \cite{BFZ}. 

\begin{theorem}\label{specnrm2mnchar}  Let $\cT=[t_{ijk}]\in \R^2\otimes\R^m\otimes \R^n$, where $m,n\ge 2$.  Suppose
$C=S(F_1) +\bi  S(F_2)= \begin{bmatrix} 0_{m\times m} & F_1 + \bi F_2 \cr F_1^* - \bi F_2^* & 0_{n\times n} \cr\end{bmatrix}$ with 
$F_1 = (t_{1jk}), F_2 = (t_{2jk}) \in \C^{m\times n}$. Then
\begin{equation}\label{specnrm2mnchar1}
\|\cT\|_\sigma=r(C) \quad\hbox{ and } \quad  \|\cT\|_*=\frac{1}{2}r^\vee(C).
\end{equation}

In particular, if $\cT\in \Q^2\otimes\Q^m\otimes \Q^n$ then  $\|\cT\|_\sigma$ and $\|\cT\|_*$ are  polynomially computable. 
\end{theorem}
To prove this theorem we need the following lemma:
\begin{lemma}\label{rveesCchar} For 
$$C=\begin{bmatrix}0_{m\times m}&T\\T^\top&0_{n\times n}\end{bmatrix}, \quad T\in\C^{m\times n}$$  
the following equality holds:
\begin{equation}\label{rveesCchar1}
\begin{aligned}
r^\vee(C)=\max\{\Re\,\tr C D^*: \,r(D)\le 1,\\
D=\begin{bmatrix}0_{m\times m}&U\\U^\top&0_{n\times n}\end{bmatrix}, U\in\C^{m\times n}\}.
\end{aligned}
\end{equation}
\end{lemma}
\begin{proof}
Let 
\begin{equation*}
\begin{aligned}
E=\begin{bmatrix}E_1&E_2\\E_3&E_4\end{bmatrix},  \quad G=\begin{bmatrix}0&E_2\\E_3&0\end{bmatrix},  \quad D=\frac{1}{2} (G+G^\top),\\
E_1\in\C^{m\times m},  \quad 
 E_2\in\C^{m\times n},E_3\in\C^{n\times m}, E_4\in \C^{n\times n}.
\end{aligned}
\end{equation*}
As $R=\diag(I_m,-I_n)\in \C^{(m+n)\times (m+n)}$ is unitary we have $r(E)=r(-RER^*)$.  Recall that $r(E)=r(E^\top)$.
Hence
\begin{equation*}
\begin{aligned}
r(G)=r(\frac{1}{2}(E-RER^*))\le r(E)\Rightarrow r(D)=r(\frac{1}{2}(G+G^\top))\le r(G)\le r(E).
\end{aligned}
\end{equation*}
Recall that  $r^\vee(C)=\max\{\Re\tr CE^*: r(E)\le 1\}$.
Observe that
$$\tr CE^*=\tr CG^*= \tr(CG^*)^\top = \tr C^\top \bar G= \tr CD^*.$$ 
This proves \eqref{rveesCchar1}.
\end{proof}

\textbf{Proof of Theorem \ref{specnrm2mnchar}.}
 Recall that for $F\in\R^{m\times n}$ one has the equality
$$\|F\|=\|S(F)\|=\lambda_{\max}(S(F))=-\lambda_{\min}(S(F)).$$
Use \eqref{maxeq}, Corollary \ref{maxcharpenc} and the above equality to deduce that
\begin{equation*}
\|\cT\|_\sigma =\max_{\theta\in[0,2\pi)} \|\cos\theta F_1+\sin\theta F_2\|=\max_{\theta\in[0,2\pi)} \lambda_1(\cos\theta S(F_1)+\sin\theta S(F_2)).
\end{equation*}
Use Lemma \ref{knwnr(C)} and Theorem \ref{ptcrC} to deduce the proposition for $\|\cT\|_\sigma$.

We now discuss $\|\cT\|_\sigma^\vee$.   Identify $\R^2\otimes \R^m\otimes \R^n$ with $\R^{m\times (2n)}$.   That is, we
identify $\cT\in \R^2\otimes \R^m\otimes \R^n$ with $\tilde F=[F_1\,F_2]\in \R^{m\times (2n)}$.  
Note that the inner product on $\R^2\otimes \R^m\otimes \R^n$ is equal to the standard inner product on $\R^{m\times (2n)}$:
\begin{equation*}
\langle \cF,\cG\rangle=\sum_{i=j=k=1}^{2,m,n} t_{ijk}g_{ijk}=\tr \tilde F \tilde G^\top.
\end{equation*}

Define a norm $\R^{m\times (2n)}$ by the equality 
$$\nu(\tilde F)=\|\cT\|_\sigma=r(S(F_1)+\bi S(F_2)).$$
The definition of the dual norm $\nu^\vee(\cdot)$ yields \cite{FL16}:
\begin{equation}\label{chardnu}
\nu^\vee(\tilde F)=\max_{\nu(\tilde G)\le 1} \tr \tilde F \tilde G^\top=\max_{\nu(\tilde G)\le 1} \tr (F_1 G_1^\top+F_2 G_2^\top)=\|\cT\|_\sigma^\vee.
\end{equation}
We now compare the above characterization with $r^\vee(C)$, where $C$ as in Theorem \ref{specnrm2mnchar}, and $T=F_1+\bi F_2$.  Use characterization \eqref{rveesCchar1}  with $U=U_1+\bi U_2, U_1,U_2\in \R^{m\times n}$, and the first equality of \eqref{specnrm2mnchar1} to deduce that
\begin{eqnarray*}
r^\vee(C)&=&\max\{2\Re TU^*: r\left(\begin{bmatrix} 0 & U \cr U^\top & 0\cr\end{bmatrix}\right)\le 1\}\\
&=& \max\{2\tr(F_1U_1^\top +F_2U_2^\top): \nu(\tilde U)\le 1\}.
\end{eqnarray*}
Compare that with \eqref{chardnu} to deduce the second equality of \eqref{specnrm2mnchar1}.\qed

We remark that in view of NP-Hardness of computation of the spectral norm of $3$-tensor $\cT$ \cite{HL13}, there is no simple SDP characterization of the spectral norm for general $\cT\in \F^l\otimes\F^m\otimes\F^n$.  
\section*{Acknowledgment}
We thank Michael Overton for useful remarks.
The work of both authors is partially supported by the Collaboration Grant for Mathematicians of the Simons Foundation
(Grant numbers: 580037 and 851334).


\begin{thebibliography}{MMM}
\bibitem{AF21} M. Aliabadi and S. Friedland, On the Complexity of Finding Tensor Ranks, \emph{Commun. Appl. Math. Comput.}, Vol. 3 (2) (2021),  281-289.
\bibitem{And73} T. Ando,  Structure of operators with numerical radius one,
Acta Sci. Math. (Szeged) 34 (1973), 11–15.
\bibitem{BFZ} Wojciech Bruzda, Shmuel Friedland, Karol {\.Z}yczkowski, Tensor rank and entanglement of pure quantum states,  {\it Linear and Multilinear Algebra} 65 pages, https://doi.org/10.1080/03081087.2023.2211717
\bibitem{deklerk_vallentin}
E. de~Klerk and F. Vallentin, On the {T}uring model complexity of interior
  point methods for semidefinite programming, {\it SIAM J. Optim. } 26~(3),
 2016,  1944--1961.
\bibitem{Fri15}  S. Friedland, \emph{Matrices: Algebra, Analysis and Applications}, World Scientific, 596 pp., 2015, Singapore, http://www2.math.uic.edu/$\sim$friedlan/bookm.pdf
\bibitem{FL16} S. Friedland and L.-H. Lim,The computational complexity of duality,  \emph{SIAM Journal on Optimization}, 26, No. 4 (2016), 2378--2393.
\bibitem{FL18}S. Friedland and L.-H. Lim,  Nuclear norm of higher-order tensors, \emph{Mathematics of Computation}, 87 (2018), 1255--1281.
\bibitem{FW20} S. Friedland and L. Wang, Spectral norm of a symmetric tensor and its computation, ,  \emph{Mathematics of Computation}, 89 (2020)., 2175--2215.
\bibitem{GM12}  B. G\"artner and J.  Matou\v{s}ek,  Approximation algorithms and semidefinite programming, Springer, Heidelberg, 2012, xii+251 pp.
\bibitem{GLS81}  M. Gr{\"o}tschel,  L. Lov\'asz and A. Schrijver,  The ellipsoid method and its consequences in combinatorial optimization, \emph{Combinatorica} 1 (1981), no. 2, 169-197. 
\bibitem{HL13} C. J. Hillar and L.-H. Lim, Most tensor problems are NP-hard, \emph{J. ACM} 60 (2013), no. 6, Art. 45, 39.
\bibitem{HJ91} R. A. Horn and C. R. Johnson, Topics in Matrix Analysis, Cambridge University Press, Cambridge, UK, 1991.
\bibitem{HJL23} H. Hu, B. Jiang, Z. Li,  Complexity and computation for the spectral norm and nuclear norm of order three tensors with one fixed dimension,
arXiv:2212.14775.
\bibitem{Ja79} J. J\'aJ\'a, Optimal evaluation of pairs of bilinear forms, \emph{SIAM J. Comput.} 8 (1979), 443--461.
\bibitem{JKLPS20}H. Jiang, T. Kathuria, Y. T. Lee, S. Padmanabhan, Z. Song,  A Faster Interior Point Method for Semidefinite Programming, arXiv:2009.10217
\bibitem{Kar84} N. Karmarkar, A new polynomial-time algorithm for linear programming,  Combinatorica 4 (4) (1984), 373-395.
\bibitem{Kha79} L. Khachiyan, A polynomial time algorithm in linear programming, Soviet Math. Dokl., 20 (1979), 191-194.
\bibitem{Kip52}  R. Kippenhahn, {\"U}ber den Wertevorrat einer Matrix, \emph{Math. Nachr.,} 6:193-228, 1951.
\bibitem{LO20} A.S. Lewis and M.L. Overton,
Partial Smoothness of the Numerical Radius at Matrices whose Fields of Values are Disks, {\it SIAM J. Matrix Anal. Appl. }41 (2020), pp. 1004–1032.
\bibitem{Mat93} R. Mathias, Matrix completions, norms and Hadamard products, {\it Proc. Amer.  Math. Soc.}, 117(4):905-918, 1993.
\bibitem{Mit23} T.  Mitchell, Convergence rate analysis and improved iterations for numerical radius computation,  SIAM J. Sci. Comput. 45 (2023), no. 2, A753-A780.
\bibitem{MO23}T. Mitchell and M. L. Overton, An Experimental Comparison of Methods for Computing the Numerical Radius,  arXiv:2310.04646.
\bibitem{NN94} Yu. Nesterov and A. S. Nemirovski, Interior-Point Polynomial Algorithms in Convex Programming, {\it Stud. Appl. Math.}, SIAM, Philadelphia, 1994.
\bibitem{PK97} L. Porkolab and L. Khachiyan,   On the complexity of semidefinite programs,  J. Global Optim. 10 (1997), no. 4, 351-365.
\bibitem{RFP10} B. Recht, M. Fazel and P.A. Parrilo,  Guaranteed minimum-rank solutions of linear matrix equations via nuclear norm,  SIAM Rev. 52 (2010), no. 3, 471-501.
 \bibitem{Ren01} J. Renegar, {\it A Mathematical View of Interior-Point Methods}, in Convex Optimization, MOS-SIAM Ser. Optim., SIAM, Philadelphia, 2001.
\bibitem{VB96} L. Vandenberghe and S. Boyd, Semidefinite Programming, SIAM Review 38, March 1996, pp. 49-95.
\bibitem{Wan23} J. Wang,  A more efficient reformulation of complex SDP as real SDP,  arXiv:2307.11599, 2023.
\bibitem{YN76} D. Yudin and A. S. Nemirovski, Informational complexity and effective methods of solution of convex extremal problems, Econom. Math. Methods, 12 (1976), 357-369.
 \end{thebibliography}
\end{document}